\documentclass[12pt, two side,emlines]{amsart}
\usepackage{amssymb,amsfonts,amssymb,amsthm,latexsym,xy,eucal,mathrsfs, graphicx, tikz, caption, hyperref}
\usepackage[misc]{ifsym}
\textwidth=15.2cm \textheight=24cm \theoremstyle{plain}
\newtheorem{lem}{Lemma}[section]

\newtheorem{thm}[lem]{Theorem}
\newtheorem{prop}[lem]{Proposition}
\newtheorem{coro}[lem]{Corollary}

\theoremstyle{definition}
\newtheorem{exa}[lem]{Example}
\newtheorem{rem}[lem]{Remark}

\newtheorem{defn}[lem]{Definition}

\numberwithin{equation}{section}  \voffset
-55truept \hoffset -15truept
\begin{document}
	\baselineskip 18truept
	\title{On the idempotent graph of matrix ring}
	\author[A.A. Patil, P.S. Momale and C.M. Jadhav]%
	{Avinash Patil$^{1,a}$, P.S. Momale$^b$ and C.M. Jadhav$^c$}	
	\address{\rm $A$ Department of Mathematics, JET's Z. B. Patil College, Dhule-424 002, India.}
	\email{\emph{avipmj@gmail.com}}
	\address{\rm	
	$B$ Department of Mathematics, P. V. P. College, Pravaranagar-413713, India}
	\email{\emph{psmomale@gmail.com}}
		\address{\rm	
		$C$ Department of Mathematics, S.V.S Dadasaheb Rawal College, Dhondaicha, India}
		\email{\emph{jchandrakant65@gmail.com}}

	\maketitle
	
	 \footnotetext[1]{Corresponding author}
	 
	\begin{abstract}Let $\mathbb{F}$ be a finite field and $R = M_2(\mathbb{F})$ be $2 \times 2$ matrix ring over $\mathbb{F}$. In this paper,
		we explicitly determine all the idempotents in $R$. Using these idempotents, we study the
		idempotent graph of $R$ whose vertex set is the set of non-trivial idempotents in $R$ and two
		idempotents $e, f$ are adjacent if $ef = 0$ or $fe = 0$. It is proved that the idempotent graph of
		$R$ is connected regular graph with diameter 2. Its girth is also characterized. Further, we determine the Wiener and Harary index of the idempotent graph
		of $R$. 
\end{abstract}
	\maketitle

\section{Introduction}
All the rings in this paper are associative, having unity and all graphs are simple. An element $h\in R$ such that $h^2 = h$ is an \textit{idempotent} and it is a \textit{central idempotent} if it commutes with each element of $R$. Two idempotents $h$ and $k$ are {\it orthogonal} if $hk = kh = 0$. In any ring with unity, 0 and 1 are idempotents called as trivial idempotents. Let $Id(R)$ be the set of idempotents in $R$. 
In 1988, Beck \cite{10} introduced the zero-divisor graph $\Gamma(R)$ of a commutative ring $R$ and conjectured that $\Gamma(R)$ is weakly perfect whenever $\omega(\Gamma(R)) < \infty$. He proved that reduced rings and principal ideal rings are the ones for which the conjecture is true. However, Anderson et al. \cite{ddn} gave a counterexample of a commutative local ring  for which the
conjecture is not true.  In \cite{7}, Anderson et al. modified Beck's definition of zero-divisor graph of a commutative ring $R$ to the now standard definition: $\Gamma(R)$ is the simple graph with vertices the nonzero zero-divisors of $R$, and vertices $x$ and $y$ are adjacent if $xy=0$. 

Cvetko-Vah et al. \cite{cd} assigned a simple graph $G(R)$ to $R$ whose vertex set is $Id(R)$, and two vertices $e$ and $f$ are adjacent if and only if :
\begin{enumerate}
	\item $ef = f e = 0$, and
	\item  $eRf \ne  0$ or $fRe\ne 0$.
\end{enumerate}
It is evident from the second condition that if the idempotents of $R$ are central, then $G(R)$ has no edges. Akbari et al. \cite{sa} defined  the {\it idempotent graph} $I(R)$ of a ring $R$ as the graph whose vertices are the nontrivial idempotents of $R$, and two distinct vertices $h$ and $k$ are adjacent if and only if $hk=kh=0$.  Observe that $h+k$ is an idempotent of $R$ whenever $h$ and $k$ are orthogonal idempotents, which is a notable algebraic property. Clearly, for a commutative ring, $I(R)$ is a subgraph of $\Gamma(R)$. The interplay between the algebraic properties of $R$ and graph-theoretic properties of $I(R)$ has been studied in \cite{sa, cd}. For example, if $D$ is a division ring, then the clique number of
$I(M_n (D)) (n\geq 2)$ is $n$, and for any commutative Artinian ring $R$ the clique number and
the chromatic number of $I(R)$ are equal to the number of maximal ideals of $R$. Also, for a division ring $D$, it proved that $diam(I(M_n (D))) = 4$ for all natural numbers $n\geq 4$ and
$diam(I(M_3 (D))) = 5$. Patil et al. \cite{ap} studied the weak perfectness of $I(R)$ and gave its applications to zero-divisor graphs.

 Let $R = M_2(\mathbb{F})$ be $2 \times 2$ matrix ring over $\mathbb{F}$, where  $\mathbb{F}$ is a finite field. In this paper, we explicitly determine all the idempotents in $R$ in terms of elements of $\mathbb{F}$. Using these idempotents, we study the variation of 
idempotent graph of $R$ whose vertex set is the set of non-trivial idempotents in $R$ and two idempotents $e, f$ are adjacent if $ef = 0$ or $fe = 0$. It is prove that the idempotent graph of
$R$ is connected regular graph with diameter 2. Its girth is also characterized. Further, we determine the Wiener index and Harary index of the idempotent graph of $R$.

We begin with the necessary concepts and terminology. For the vertices $a$ and $b$  of a graph $G$, the \textit{distance} $d(a,b)$ between $a$ and $b$ is the number
of edges in the shortest path between $a$ and $b$. The largest
distance among all distances between pairs of the vertices of
a graph $G$ is the \textit{diameter} of $G$, denoted by $diam(G)$.
A graph $G$ is \textit{connected} if for any vertices $x$ and $y$ of $G$ there is a path between $x$ and $y$. For $x\in V(G)$-the set
of vertices of $G$, the set of neighbors of $x$ in $G$ is denoted
by $N(x) = \{y \in V(G) ~|~ y \textnormal{ is adjacent to } x \textnormal{ in } G\}$. The \textit{girth}
of $G$ is the length of the shortest cycle in $G$ and is denoted
by $gr(G)$. If $G$ has no cycles, then the girth of $G$ is infinite. The \textit{degree} of a vertex $v$ in $G$, denoted by $d(v)$, is the number of vertices adjacent to $v$ in $G$. A graph $G$ is {\it regular} if each of its vertex have the same degree. A graph is {\it complete} if any two of its vertices are adjacent. The complete graph on $m$ vertices is denoted by $K_m$. Henceforth we use $a\sim b$ to denote the vertices $a$ and $b$ are adjacent, $\mathbb{F}$ to denote a finite field with $n=|\mathbb{F}|$, $\boldsymbol{0}$ to denote $2\times 2$ zero matrix and $\boldsymbol{1}$ to denote $2\times 2$ identity matrix.

\section{Idempotent in Matrix ring} 
First we determine the number of idempotents in $R$.
\begin{thm}\label{t1}
Let $R=M_2(\mathbb{F})$. Then $R$ contains exactly $n^2+n+2$ idempotents. 
\end{thm}
 \begin{proof}Let $R=M_2(\mathbb{F})$, where $\mathbb{F}$ is a finite field.  Then $R$ is vector space over $\mathbb{F}$. Let $A=\left[\begin{array}{lr}
 a & b\\
 c & d
 \end{array}\right]$ be an idempotent in $R$, i.e $A^2=A$, which gives $(\det A)^2= \det A$, i.e. $\det A$ is an idempotent in $\mathbb{F}$. Hence $\det A\in \{0, 1\}$. If $\det A = 1$, then $A$ is invertible in $R$. This together with $A^2=A$ gives $A=\left[\begin{array}{lr}
  1 & 0\\
  0 & 1
  \end{array}\right]$. Suppose that $\det A =0$. Then $ad-bc=0$. If $A$ is non-zero, then $A^2-A=\boldsymbol{0}$, which gives $tr(A)=1$, i.e. $a+d=1$ in $\mathbb{F}$. Then $ad=bc$ gives $a(1-a)=bc$. Now there are following two cases.\\
  {\bf Case 1.} Since $a(1-a)=bc$, we must have either $b=0$ or $c=0$ if and only if $a\in \{0, 1\}$. Hence $A$ is of the form $\left[\begin{array}{lr}
0 & 0\\
c & 1
\end{array}\right],~ \left[\begin{array}{lr}
 0 & b\\
 0 & 1
 \end{array}\right],~ \left[\begin{array}{lr}
 1 & 0\\
 c & 0
 \end{array}\right],~ \left[\begin{array}{lr}
1 & b\\
0 & 0
\end{array}\right]$, for some $b,c \in \mathbb{F}$. In this case the number of choices for distinct $A= n+(n-1)+n+(n-1)=4n-2$.\\
{\bf Case 2.}  Suppose that $b\ne 0$ and $c\ne 0$, which yields $a\notin \{0,1\}$. Then $a(1-a)=bc$ gives $c=a(1-a)b^{-1}$. Hence $A=\left[\begin{array}{lr}
 a & b\\
 a(1-a)b^{-1} & 1-a
 \end{array}\right]$.  Moreover, for each $a\in \mathbb{F}\backslash \{0,1\}$ and non-zero $b\in \mathbb{F}$, we get an idempotent in $R$. Hence the number of idempotents in this case = $(n-2)(n-1)$.
 
 Thus the total number of idempotents in $R$ (including trivial idempotents)= $4n-2+ (n-2)(n-1)+2 = n^2+n+2$. 
\end{proof}
Recently Masaklar et al. \cite{mklk} also determined the idempotents in the matrix rings.

 \noindent {\bf Notation:} Let $E_{ij}$ be the matrix units, where $i,j\in \{1,2\}$, i.e., $E_{ij}=[e_{kl}]_{2\times 2}$, where\\ $e_{kl}=\left\{\begin{array}{cl}
  	1, & \textnormal{ if } k=i,~ l=j\\
  	0, & \textnormal{ otherwise } 
  \end{array}\right.$ Then we can write $A=\left[\begin{array}{lr}
  a & b\\
  c & d
\end{array}\right]=aE_{11}+bE_{12}+cE_{21}+dE_{22}$.
  \begin{rem}\label{r1}
  We get a partition of $Id(R)$ as $\displaystyle{ Id(R)=P_0\dot{\cup}P_1\dot{\cup}\cdots \dot{\cup}P_7}$, where $$ P_0=\left\{\left[\begin{array}{lr}
 0 & 0\\
 0 & 0
 \end{array}\right], \left[\begin{array}{lr}
  1 & 0\\
  0 & 1
  \end{array}\right]\right\},  P_1=\left\{E_{11}\right\},~P_2=\left\{E_{22}\right\}, $$  $$P_3=\left\{\left[\begin{array}{lr}
0 & 0\\
c & 1
\end{array}\right]=cE_{21}+E_{22}~|~\textnormal{ for }\\ c \in \mathbb{F}\backslash\{0\}\right\}, $$ $$P_4= \left\{\left[\begin{array}{lr}
0 & b\\
0 & 1
\end{array}\right]=bE_{12}+E_{22}~| \textnormal{ for } b \in \mathbb{F}\backslash\{0\}\right\},$$ $$P_5=\left\{\left[\begin{array}{lr}
1 & 0\\
c & 0
\end{array}\right]=E_{11}+cE_{21}~|\textnormal{ for } c \in \mathbb{F}\backslash\{0\}\right\},$$ $$P_6= \left\{\left[\begin{array}{lr}
 1 & b\\
 0 & 0
 \end{array}\right]=E_{11}+bE_{12}~ |\textnormal{ for } b \in \mathbb{F}\backslash\{0\}\right\}$$ and 
$$P_7=\left\{\left[\begin{array}{lr}
a & b\\
a(1-a)b^{-1} & 1-a
\end{array}\right]~| \textnormal{ for }  a\in \mathbb{F}\backslash\{0,1\} \textnormal{ and  } b \in \mathbb{F}\backslash\{0\}\right\}.$$ Observe that $\left[\begin{array}{lr}
a & b\\
a(1-a)b^{-1} & 1-a
\end{array}\right]= aE_{11}+bE_{12}+(a(1-a)b^{-1})E_{21}+(1-a)E_{22}$.
Also,  $|P_1|=|P_2|=1$, $|P_3|=|P_4|=|P_5|=|P_6|= n-1$ and $|P_7|=(n-2)(n-1)$.
  \end{rem}
  

\section{Variation of Idempotent graph}
Akbari et al. \cite{sa} introduced  the {\it idempotent graph} $I(R)$ of a ring $R$ as the graph whose vertices are the nontrivial idempotents of $R$, and two distinct vertices $h$ and $k$ are adjacent if and only if $hk=kh=0$.

We consider the following variation of the idempotent graph. 
\begin{defn}
	Let $R$ be a ring. We assign a graph $G_{Id}(R)$ to $R$ whose vertices are the nontrivial idempotents of $R$, and two distinct vertices $h$ and $k$ are adjacent if and only if $hk=0$ or $kh=0$. 
\end{defn}
\noindent {\bf Note  :} Let $R$ be a ring $R$. Then the graphs $I(R)$ and $G_{Id}(R)$ have the same vertex set. In fact,  $I(R)$ is a subgraph of $G_{Id}(R)$. Moreover, if $R$ is an abelian ring, then the two graphs are identical.
\begin{exa}
	Let $R=M_2(\mathbb{Z}_2)$. Then $Id(R)^*=Id(R)\backslash \{\boldsymbol{0},\boldsymbol{1}\}=\{e_1,e_2,\cdots, e_6\}$, where $e_1=\left[\begin{array}{lr}
		1 & 0\\
		0 & 0
	\end{array}\right], e_2=\left[\begin{array}{lr}
		0 & 0\\
		0 & 1
	\end{array}\right], e_3= \left[\begin{array}{lr}
		0 & 0\\
		1 & 1
	\end{array}\right], e_4=\left[\begin{array}{lr}
		0 & 1\\
		0 & 1
	\end{array}\right], e_5=\left[\begin{array}{lr}
		1 & 1\\
		0 & 0
	\end{array}\right]$ and $e_6=\left[\begin{array}{lr}
		1 & 0\\
		1 & 0
	\end{array}\right]$. The graphs $I(R)$ and $G_{Id}(R)$ are as depicted in Figure \ref{f1}.

\begin{figure}
	\begin{center}
	\begin{tikzpicture}
		\draw(0,0)--(0,1); \draw(-1,0)--(-1,1); \draw(1,0)--(1,1);
				
		\draw[fill=black](0,0)circle(.03);
		\draw[fill=black](0,1)circle(.03);
		\draw[fill=black](-1,1)circle(.03);
		\draw[fill=black](-1,0)circle(.03);
		\draw[fill=black](1,0)circle(.03);
		\draw[fill=black](1,1) circle(.03);
		
		\node[below] at(-1,0){$e_4$};
		\node[above] at(-1,1){$e_5$};
		\node[below] at(0,0){$e_1$};
		\node[above] at(0,1){$e_2$};
		\node[below] at(1,0){$e_3$};
		\node[above] at(1,1){$e_6$};
		\node at(0,-1){$I(R)$};
		
			\draw(6,0)--(6,1); \draw(6,0)--(5,0)--(5,1)--(6,1); \draw(6,0)--(7,0)--(7,1)--(6,1);
			
			\draw (7,1) arc
			[start angle=0,	end angle=180,
			x radius=1cm,y radius =.5cm
			];
			
			\draw (7,0) arc
			[start angle=0,	end angle=-180,
			x radius=1cm,y radius =0.5cm
			];
		
		\draw[fill=black](6,0)circle(.03);
		\draw[fill=black](6,1)circle(.03);
		\draw[fill=black](5,1)circle(.03);
		\draw[fill=black](5,0)circle(.03);
		\draw[fill=black](7,0)circle(.03);
		\draw[fill=black](7,1) circle(.03);
		
		\node[left] at(5,0){$e_4$};
		\node[left] at(5,1){$e_5$};
		\node[below] at(6,0){$e_1$};
		\node[above] at(6,1){$e_2$};
		\node[right] at(7,0){$e_3$};
		\node[right] at(7,1){$e_6$};
		\node at(6,-1){$G_{Id}(R)$};
		
	\end{tikzpicture}\caption{Idempotent graph of $M_2(\mathbb{Z}_2)$ and its variation}\label{f1}
\end{center}
\end{figure}
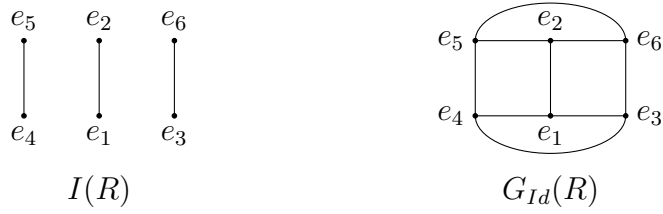
\end{exa}
\noindent{\bf Note:} Since 
$\left[\begin{array}{lr}
	x & y\\
	z & w
\end{array}\right]\left[\begin{array}{lr}
	0 & 0\\
	c & 1
\end{array}\right]=\boldsymbol{0}$ if and only if $\left[\begin{array}{lr}
	0 & c\\
	0 & 1
\end{array}\right]\left[\begin{array}{lr}
	x & y\\
	z & w
\end{array}\right]^t=\boldsymbol{0}$ and\\ $\left[\begin{array}{lr}
	0 & 0\\
	c & 1
\end{array}\right]\left[\begin{array}{lr}
	x & y\\
	z & w
\end{array}\right]=\boldsymbol{0}$ if and only if $\left[\begin{array}{lr}
	x & y\\
	z & w
\end{array}\right]^t\left[\begin{array}{lr}
	0 & c\\
	0 & 1
\end{array}\right]=\boldsymbol{0}$, $deg\left(\left[\begin{array}{lr}
	0 & b\\
	0 & 1
\end{array}\right]\right)=deg\left(\left[\begin{array}{lr}
	0 & 0\\
	c & 1
\end{array}\right]\right)$. Similarly, $deg\left(\left[\begin{array}{lr}
	1 & 0\\
	c & 0
\end{array}\right]\right)=deg\left(\left[\begin{array}{lr}
	1 & b\\
	0 & 0
\end{array}\right]\right)$.

 Now we will determine the paths and distance between every pair of elements of the partitioning sets given Remark \ref{r1}. 
 
 \begin{lem}\label{l1}Let $R=M_2(\mathbb{F})$, where $\mathbb{F}$ is a finite field and $P_0,\cdots, P_7$ as in Remark \ref{r1}. Then the following statements hold in $G_{Id}(R)$.
 	\begin{enumerate}
 		\item $\displaystyle{d(E_{11},f)=\left\{\begin{array}{cl}
 			1, & \textnormal{ if } f\in \cup_{i=2}^4P_i\\
 			2, & \textnormal{ if } f\in \cup_{i=5}^7P_i
 		\end{array}\right.}$ and $deg(E_{11})=2n-1$.
 	\item $\displaystyle{d(E_{22},f)=\left\{\begin{array}{cl}
 		1, & \textnormal{ if } f\in P_1\cup P_5\cup P_6\\
 		2, & \textnormal{ if } f\in P_3\cup P_4\cup P_7
 	\end{array}\right.}$ and  $deg(E_{22})=2n-1$.
 \item Let $E=\left[\begin{array}{lr}
 	a & b\\
 	a(1-a)b^{-1} & 1-a
 \end{array}\right]\in P_7$ and $A$ is a nontrivial idempotent not in $P_7$. Then $E$ and  $A$ are adjacent if and only if\\ $A\in \mathbf{A}=\left\{ \left[\begin{array}{lr}
 0 & 0\\
 (a-1)b^{-1} & 1
\end{array}\right],~\left[\begin{array}{lr}
0 & -a^{-1}b\\
0 & 1
\end{array}\right],~\left[\begin{array}{lr}
1 & (a-1)^{-1}b\\
0& 0
\end{array}\right], \left[\begin{array}{lr}
1 & 0\\
-ab^{-1} & 0
\end{array}\right]\right\}$.
\item Let $A_1=\left[\begin{array}{lr}
	0 & 0\\
	c & 1
\end{array}\right]\in P_3$. 
\begin{enumerate}
	\item[(a)] If
 $\displaystyle{B_1\in \cup_{i=3}^6P_i}$, then\\
 $\displaystyle{d(A_1,B_1)=\left\{\begin{array}{cl}
	1, & \textnormal{ if } B_1\in P_5\cup \{x_1\}\\
	2, & \textnormal{ if } B_1\in \left(P_3\cup P_4\cup P_6\right)\backslash\{x_1,A_1\}
\end{array}\right.}$, where $x_1=\left[\begin{array}{lr}
0 & -c^{-1}\\
0 & 1
\end{array}\right]$.
\item[(b)] If $B_1=\left[\begin{array}{lr}
	a & b\\
	a(1-a)b^{-1} & 1-a
\end{array}\right]\in P_7$, then\\
$d(A_1,B_1)=\left\{\begin{array}{cl}
1, & \textnormal{ if } B_1=x_2\\
2, & \textnormal{ if } B_1\in P_7\backslash\{x_2\}
\end{array}\right.$, where $x_2=\left[\begin{array}{lr}
a & -(1-a)c^{-1}\\
-ac & 1-a
\end{array}\right]$.
\end{enumerate}
\item Let $A_2=\left[\begin{array}{lr}
	0 & c\\
	0 & 1
\end{array}\right]\in P_4$. 
\begin{enumerate}
	\item[(a)] If
	$\displaystyle{B_2\in \cup_{i=3}^6P_i}$, then\\
	$\displaystyle{d(A_2,B_2)=\left\{\begin{array}{cl}
		1, & \textnormal{ if } B_2\in P_6\cup \{x_3\}\\
		2, & \textnormal{ if } B_2\in \left(P_3\cup P_4\cup P_5\right)\backslash\{x_3,A_2\}
	\end{array}\right.}$, where $x_3=\left[\begin{array}{lr}
		0 & 0\\
		-c^{-1} & 1
	\end{array}\right]$.
	\item[(b)] If $B_2=\left[\begin{array}{lr}
		a & b\\
		a(1-a)b^{-1} & 1-a
	\end{array}\right]\in P_7$, then\\
	$d(A_2,B_2)=\left\{\begin{array}{cl}
		1, & \textnormal{ if } B_2=x_4\\
		2, & \textnormal{ if } B_2\in P_7\backslash\{x_4\}
	\end{array}\right.$, where $x_4=\left[\begin{array}{lr}
	a & -ac\\
	(a-1)c^{-1} & 1-a
\end{array}\right]$.
\end{enumerate}
\item Let $A_3=\left[\begin{array}{lr}
	1 & 0\\
	c & 0
\end{array}\right]\in P_5$. 
\begin{enumerate}
	\item[(a)] If
	$\displaystyle{B_3\in \cup_{i=3}^6P_i}$, then\\
	$\displaystyle{d(A_3,B_3)=\left\{\begin{array}{cl}
		1, & \textnormal{ if } B_3\in P_3\cup \{x_5\}\\
		2, & \textnormal{ if } B_3\in \left(P_4\cup P_5\cup P_6\right)\backslash\{x_5,A_3\}
	\end{array}\right.}$, where $x_5=\left[\begin{array}{lr}
		1 & -c^{-1}\\
		0 & 0
	\end{array}\right]$.
	\item[(b)] If $B_3=\left[\begin{array}{lr}
		a & b\\
		a(1-a)b^{-1} & 1-a
	\end{array}\right]\in P_7$, then\\
	$d(A_3,B_3)=\left\{\begin{array}{cl}
		1, & \textnormal{ if } B_3=x_6\\
		2, & \textnormal{ if } B_3\in P_7\backslash\{x_6\}
	\end{array}\right.$, where $x_6=\left[\begin{array}{lr}
	a & -ac^{-1}\\
	(a-1)c & 1-a
\end{array}\right]$.
\end{enumerate}
\item Let $A_4=\left[\begin{array}{lr}
	1 & c\\
	0 & 0
\end{array}\right]\in P_6$. 
\begin{enumerate}
	\item[(a)] If
	$\displaystyle{B_4\in \cup_{i=3}^6P_i}$, then\\
	$\displaystyle{d(A_4,B_4)=\left\{\begin{array}{cl}
		1, & \textnormal{ if } B_4\in P_4\cup \{x_7\}\\
		2, & \textnormal{ if } B_4\in \left(P_3\cup P_5\cup P_6\right)\backslash\{x_7,A_4\}
	\end{array}\right.}$, where $x_7=\left[\begin{array}{lr}
		1 & 0\\
		-c^{-1} & 0
	\end{array}\right]$.
	\item[(b)] If $B_4=\left[\begin{array}{lr}
		a & b\\
		a(1-a)b^{-1} & 1-a
	\end{array}\right]\in P_7$, then\\
	$d(A_3,B_3)=\left\{\begin{array}{cl}
		1, & \textnormal{ if } B_3=x_8\\
		2, & \textnormal{ if } B_3\in P_7\backslash\{x_8\}
	\end{array}\right.$, where $x_8=\left[\begin{array}{lr}
		a & (a-1)c\\
		-ac^{-1} & 1-a
	\end{array}\right]$.
\end{enumerate}

 	\end{enumerate}
 \end{lem}
\begin{proof}
(1) Observe that $E_{11}=\left[\begin{array}{lr}
		1 & 0\\
		0 & 0
	\end{array}\right]$ is adjacent to $\left[\begin{array}{lr}
	0 & 0\\
	0 & 1
\end{array}\right],~\left[\begin{array}{lr}
0 & 0\\
c & 1
\end{array}\right]$ and $\left[\begin{array}{lr}
0 & b\\
0 & 1
\end{array}\right]$ only, where $b,c \in \mathbb{F}\backslash\{0\}$. Hence $E_{11}$ is adjacent to every element of $\displaystyle{\cup_{i=2}^4P_i}$ and $E_{11}$ is not adjacent to every element of $\displaystyle{\cup_{i=5}^7P_i}$. Also for any $f\in P_5\cup P_6$, $E_{11}\sim E_{22} \sim f$ is a path in $G_{Id}(R)$. Let $x=\left[\begin{array}{lr}
		a & b\\
		a(1-a)b^{-1} & 1-a
	\end{array}\right]\in P_7$. Then $E_{11}$ and $x$ are non-adjacent and $E_{11}\sim g \sim x$ is a path (since $E_{11}g=gx=\boldsymbol{0}$), where  $g=\left[\begin{array}{lr}
		0 & 0\\
		(a-1)b^{-1} & 1
	\end{array}\right]\in P_3$. Hence $d(E_{11},x)=2$, $\forall x\in P_7$. Therefore $\displaystyle{d(E_{11},f)=\left\{\begin{array}{cl}
		1, & \textnormal{ if } f\in \cup_{i=2}^4P_i\\
		2, & \textnormal{ if } f\in \cup_{i=5}^7P_i
	\end{array}\right.}$.\\
(2) Observe that $E_{22}=\left[\begin{array}{lr}
		0 & 0\\
		0 & 1
	\end{array}\right]$ is adjacent to is adjacent to $\left[\begin{array}{lr}
	1 & 0\\
	0 & 0
\end{array}\right],~\left[\begin{array}{lr}
1 & b\\
0 & 0
\end{array}\right]$ and $\left[\begin{array}{lr}
1 & 0\\
c & 0
\end{array}\right]$ only, where $b,c \in \mathbb{F}\backslash\{0\}$, i.e. $E_{22}$ is adjacent to every element of $P_1\cup P_5\cup P_6$ and $E_{22}$ is not adjacent to any element of $P_3\cup P_4\cup P_7$. Also, for any $f\in P_3\cup P_4$, $E_{22}\sim E_{11} \sim f$ is a path in $G_{Id}(R)$(since $E_{11}E_{22}=\boldsymbol{0}=fE_{11}$). Let $y=\left[\begin{array}{lr}
		a & b\\
		a(1-a)b^{-1} & 1-a
	\end{array}\right]\in P_7$. Then $E_{22}$ and $y$ are non-adjacent and $E_{22}\sim h \sim y$ is a path (since $E_{22}h=hy=\boldsymbol{0}$), where  $h=\left[\begin{array}{lr}
		1 & (a-1)^{-1}b\\
		0 & 0
	\end{array}\right]\in P_6$. Hence $d(E_{22},y)=2$, $\forall y\in P_7$. Therefore $d(E_{22},f)=\left\{\begin{array}{cl}
		1, & \textnormal{ if } f\in P_1\cup P_5\cup P_6\\
		2, & \textnormal{ if } f\in P_3\cup P_4\cup P_7
	\end{array}\right..$\\
(3) Observe that $E=\left[\begin{array}{lr}
	a & b\\
	a(1-a)b^{-1} & 1-a
\end{array}\right]$ is not adjacent to $E_{11}$ and $E_{22}$. Also, for nonzero $c\in \mathbb{F}$, we have $\left[\begin{array}{lr}
	1 & 0\\
	c & 0
\end{array}\right]E\ne \boldsymbol{0}$, $E\left[\begin{array}{lr}
1 & c\\
0 & 0
\end{array}\right]\ne \boldsymbol{0}$, $\left[\begin{array}{lr}
0 & c\\
0 & 1
\end{array}\right]E\ne \boldsymbol{0}$ and $E\left[\begin{array}{lr}
	0 & 0\\
	c & 1
\end{array}\right]\ne \boldsymbol{0}$. Let $\left[\begin{array}{lr}
	0 & 0\\
	c & 1
\end{array}\right]E= \boldsymbol{0}$. Which gives $\left[\begin{array}{lr}
	0 & 0\\
	ac+a(1-a)b^{-1} & bc+1-a
\end{array}\right]=\boldsymbol{0}$, which yields $ ac+a(1-a)b^{-1}=0$ and $bc+1-a=0$. Consequently, $c=(a-1)b^{-1}$. Thus  $\left[\begin{array}{lr}
	0 & 0\\
	c & 1
\end{array}\right]E= \boldsymbol{0}$ if and only if $c=(a-1)b^{-1}$. Next,  let $E\left[\begin{array}{lr}
	1 & 0\\
	c & 0
\end{array}\right]= \boldsymbol{0}$. Which gives $\left[\begin{array}{lr}
	a+bc & 0\\
	a(1-a)b^{-1}+c(1-a) & 0
\end{array}\right]=\boldsymbol{0}$ which yields $ a(1-a)b^{-1}+c(1-a)=0$ and $a+bc=0$. Consequently, $c=-ab^{-1}$. Thus  $E\left[\begin{array}{lr}
	1 & 0\\
	c & 0
\end{array}\right]= \boldsymbol{0}$ if and only if $c=-ab^{-1}$.  Similarly $E\left[\begin{array}{lr}
1 & c\\
0 & 0
\end{array}\right]E= \boldsymbol{0}$ gives $c=(a-1)^{-1}b$; and $E\left[\begin{array}{lr}
0 & c\\
0 & 1
\end{array}\right]= \boldsymbol{0}$ gives $c=-a^{-1}b$. 
 Therefore, if $A$ is a nontrivial idempotent not in $P_7$, then $E$ is adjacent to $A$ if and only if  $A\in \mathbf{A}= \left\{ \left[\begin{array}{lr}
 	0 & 0\\
 	(a-1)b^{-1} & 1
 \end{array}\right],~\left[\begin{array}{lr}
 	0 & -a^{-1}b\\
 	0 & 1
 \end{array}\right],~\left[\begin{array}{lr}
 	1 & (a-1)^{-1}b\\
 	0& 0
 \end{array}\right], \left[\begin{array}{lr}
 	1 & 0\\
 	-ab^{-1} & 0
 \end{array}\right]\right\}$.\\
(4) (a) Let $A_1=\left[\begin{array}{lr}
	0 & 0\\
	c & 1
\end{array}\right]$,
	$\displaystyle{B_1\in \cup_{i=3}^6P_i}$ and  $x_1=\left[\begin{array}{lr}
		0 & -c^{-1}\\
		0 & 1
	\end{array}\right]$. Since $A_1x_1=\boldsymbol{0}=\left[\begin{array}{lr}
	1 & 0\\
	b & 0
\end{array}\right]A_1$, for every nonzero $b\in \mathbb{F}$, we have $A_1$ adjacent to $x_1$ and to every element of $P_5$. On the other hand, $A_1$ is not adjacent to any element of $P_3\cup P_4\cup P_6$ (since 
	$A_1B_1\ne \boldsymbol{0}$ and $B_1A_1\ne \boldsymbol{0}$, for every $B_1\in (P_3\cup P_4\cup P_6)\backslash\{x_1\}$). If $B_1\in (P_3\cup P_4)\backslash\{x_1\}$, then $E_{11}$ is a common neighbour of $A_1$ and $B_1$. If $B_1=\left[\begin{array}{lr}
		1 & b\\
		0 & 0
	\end{array}\right]\in P_6$, then $A_1\sim y\sim B_1$ is a path(since $yA_1=B_1y=\boldsymbol{0}$), where $y=\left[\begin{array}{lr}
	1 & 0\\
	-b^{-1} & 0
\end{array}\right]\in P_5$. Hence the result.\\
(b)  Let $B_1=\left[\begin{array}{lr}
		a & b\\
		a(1-a)b^{-1} & 1-a
	\end{array}\right]\in P_7$. Then $B_1A_1\ne \boldsymbol{0}$, since $b$ is nonzero. Let $A_1B_1=\boldsymbol{0}$, i.e.,  $\left[\begin{array}{lr}
		0 & 0\\
		c & 1
	\end{array}\right]\left[\begin{array}{lr}
		a & b\\
		a(1-a)b^{-1} & 1-a
	\end{array}\right]= \boldsymbol{0}$. Which gives\\ $\left[\begin{array}{lr}
	0 & 0\\
	ac+a(1-a)b^{-1} & bc+1-a
\end{array}\right]= \boldsymbol{0}$ if and only if $b=(a-1)c^{-1}$, i.e. $B_1=x_2$. Suppose that $B_1\ne x_2$. Then $A_1\sim y_1\sim B_1$ is a path (since $y_1A_1=B_1y_1=\boldsymbol{0}$), where  $y_1=\left[\begin{array}{lr}
1 & 0\\
-ab^{-1} & 0
\end{array}\right]$. Hence the result.\\
(5) (a) Let $A_2=\left[\begin{array}{lr}
	0 & c\\
	0 & 1
\end{array}\right]$,
$\displaystyle{B_2\in \cup_{i=3}^6P_i}$ and  $x_3=\left[\begin{array}{lr}
	0 & 0\\
	-c^{-1} & 1
\end{array}\right]$. Since $x_3A_2=\boldsymbol{0}=A_2\left[\begin{array}{lr}
1 & b\\
	0 & 0
\end{array}\right]$, for every nonzero $b\in \mathbb{F}$, we have $A_2$ adjacent to $x_3$ and to every element of $P_6$. On the other hand, $A_2$ is not adjacent to any element of $(P_3\cup P_4\cup P_5)\backslash\{x_3\}$ (since 
$A_2B_2\ne \boldsymbol{0}$ and $B_2A_2\ne \boldsymbol{0}$, for every $B_2\in (P_3\cup P_4\cup P_5)\backslash\{x_3\}$). If $B_2\in (P_3\cup P_4)\backslash\{x_3\}$, then $E_{11}$ is a common neighbour of $A_2$ and $B_2$. If $B_2=\left[\begin{array}{lr}
	1 & 0\\
	b & 0
\end{array}\right]\in P_5$, then $A_2\sim y'\sim B_2$ is a path(since $A_2y'=y'B_2=\boldsymbol{0}$), where $y'=\left[\begin{array}{lr}
	1 & -b^{-1}\\
	 0 & 0
\end{array}\right]\in P_6$. Hence the result.\\
(b)  Let $B_2=\left[\begin{array}{lr}
	a & b\\
	a(1-a)b^{-1} & 1-a
\end{array}\right]\in P_7$. Then $A_2B_2\ne \boldsymbol{0}$, since $1-a$ is nonzero. Let $B_2A_2=\boldsymbol{0}$, i.e.,  $\left[\begin{array}{lr}
a & b\\
a(1-a)b^{-1} & 1-a
\end{array}\right]\left[\begin{array}{lr}
	0 & c\\
	0 & 1
\end{array}\right]= \boldsymbol{0}$. Which gives\\ $\left[\begin{array}{lr}
	0 & ac+b\\
	0 & ac(1-a)b^{-1}+(1-a)
\end{array}\right]= \boldsymbol{0}$ if and only if $b=-ac$, i.e. $B_2=x_4$. Suppose that $B_2\ne x_4$. Then $A_2\sim y_2\sim B_2$ is a path (since $A_2y_2=y_2B_2=\boldsymbol{0}$), where  $y_2=\left[\begin{array}{lr}
	1 & (a-1)^{-1}\\
	0 & 0
\end{array}\right]$. Hence the result.\\
(6) (a) Let $A_3=\left[\begin{array}{lr}
	1 & 0\\
	c & 0
\end{array}\right]$,
$\displaystyle{B_3\in \cup_{i=3}^6P_i}$ and  $x_5=\left[\begin{array}{lr}
	1 & -c^{-1}\\
	0 & 0
\end{array}\right]$. Since $x_5A_3=\boldsymbol{0}=A_3\left[\begin{array}{lr}
	0 & 0\\
	b & 1
\end{array}\right]$, for every nonzero $b\in \mathbb{F}$, we have $A_3$ adjacent to $x_5$ and to every element of $P_3$. On the other hand, $A_3$ is not adjacent to any element of $(P_4\cup P_5\cup P_6)\backslash\{x_5\}$ (since 
$A_3B_3\ne \boldsymbol{0}$ and $B_3A_3\ne \boldsymbol{0}$, for every $B_3\in (P_4\cup P_5\cup P_6)\backslash\{x_5\}$). If $B_3\in (P_5\cup P_6)\backslash\{x_5\}$, then $E_{22}$ is a common neighbour of $A_3$ and $B_3$. If $B_3=\left[\begin{array}{lr}
	0 & b\\
	0 & 1
\end{array}\right]\in P_4$, then $A_3\sim y''\sim B_3$ is a path (since $A_3y''=\boldsymbol{0}=y''B_3$), where $y''=\left[\begin{array}{lr}
0 & 0\\
-b^{-1} & 1
\end{array}\right]\in P_3$.  Hence the result.\\
(b)  Let $B_3=\left[\begin{array}{lr}
	a & b\\
	a(1-a)b^{-1} & 1-a
\end{array}\right]\in P_7$. Then $A_3B_3\ne \boldsymbol{0}$, since $a$ is nonzero. Let $B_3A_3=\boldsymbol{0}$, i.e.,  $\left[\begin{array}{lr}
	a & b\\
	a(1-a)b^{-1} & 1-a
\end{array}\right]\left[\begin{array}{lr}
	1 & 0\\
	c & 0
\end{array}\right]= \boldsymbol{0}$. Which gives\\ $\left[\begin{array}{lr}
	a+bc & 0\\
	a(1-a)b^{-1}+c(1-a) & 0
\end{array}\right]= \boldsymbol{0}$ if and only if $b=-ac^{-1}$, i.e. $B_3=x_6$. Suppose that $B_3\ne x_6$. Then $A_3\sim y_3\sim B_3$ is a path (since $A_3y_3=y_3B_3=\boldsymbol{0}$), where  $y_3=\left[\begin{array}{lr}
0 & 0\\
(a-1)b^{-1} & 1
\end{array}\right]\in P_3$. Hence the result.\\
(7) (a) Let $A_4=\left[\begin{array}{lr}
1 & c\\
0 & 0
\end{array}\right]$,
$\displaystyle{B_4\in \cup_{i=3}^6P_i}$ and  $x_7=\left[\begin{array}{lr}
1 & 0\\
-c^{-1} & 0
\end{array}\right]$. Since $A_4x_7=\boldsymbol{0}=\left[\begin{array}{lr}
0 & b\\
0 & 1
\end{array}\right]A_4$, for every nonzero $b\in \mathbb{F}$, we have $A_4$ adjacent to $x_7$ and to every element of $P_4$. On the other hand, $A_4$ is not adjacent to any element of $(P_3\cup P_5\cup P_6)\backslash\{x_7\}$ (since 
$A_4B_4\ne \boldsymbol{0}$ and $B_4A_4\ne \boldsymbol{0}$, for every $B_4\in (P_3\cup P_5\cup P_6)\backslash\{x_7\}$). If $B_4\in (P_5\cup P_6)\backslash\{x_7\}$, then $E_{22}$ is a common neighbour of $A_4$ and $B_4$. If $B_4=\left[\begin{array}{lr}
0 & 0\\
b & 1
\end{array}\right]\in P_3$, then $A_4\sim \tilde{y}\sim B_4$ is a path (since $\tilde{y}A_4=\boldsymbol{0}=B_4\tilde{y}$), where $\tilde{y}=\left[\begin{array}{lr}
0 & -b^{-1}\\
 0 & 1
\end{array}\right]\in P_6$.  Hence the result.\\
(b)  Let $B_4=\left[\begin{array}{lr}
a & b\\
a(1-a)b^{-1} & 1-a
\end{array}\right]\in P_7$. Then $B_4A_4\ne \boldsymbol{0}$, since $a$ is nonzero. Let $A_4B_4=\boldsymbol{0}$, i.e.,  $\left[\begin{array}{lr}
1 & c\\
0 & 0
\end{array}\right]\left[\begin{array}{lr}
a & b\\
a(1-a)b^{-1} & 1-a
\end{array}\right]= \boldsymbol{0}$. Which gives\\ $\left[\begin{array}{lr}
a+ac(1-a)^{-1} & b+c(1-a)\\
0 & 0
\end{array}\right]= \boldsymbol{0}$ if and only if $b=-(1-a)c$, i.e. $B_4=x_8$. Suppose that $B_4\ne x_8$. Then $A_4\sim y_4\sim B_4$ is a path (since $y_4A_4=B_4y_4=\boldsymbol{0}$), where  $y_4=\left[\begin{array}{lr}
0 & -a^{-1}b\\
0 & 1
\end{array}\right]\in P_3$. Hence the result.
\end{proof}
\begin{coro}\label{c1}
	Let $R=M_2(\mathbb{F})$, where $\mathbb{F}$ is a finite field and $P_0,\cdots, P_7$ as in Remark \ref{r1}. Then $\deg(A)=2n-1$, for each $\displaystyle{A\in \cup_{i=1}^6P_i}$.
\end{coro}
\begin{proof}Clearly, for any vertex $x$ in a graph $G$, $\deg(x)=\left|\{y\in V(G)~|~d(x,y)=1\}\right|$. By Lemma \ref{l1}(1), we get  $\deg(E_{11})=|P_2|+|P_3|+|P_4|=1+n-1+n-1=2n-1$. Similarly, by Lemma \ref{l1}(2), $\deg(E_{22})=|P_1|+|P_5|+|P_6|=1+n-1+n-1=2n-1$.
	
Suppose that $A\in P_3$. Then  $A_1=\left[\begin{array}{lr}
	0 & 0\\
	c & 1
\end{array}\right]$, for some nonzero $c\in \mathbb{F}$. Hence $A$ is adjacent to every element of $P_5\cup \{E_{11}, x_1\}$, where  $x_1=\left[\begin{array}{lr}
	0 & -c^{-1}\\
	0 & 1
\end{array}\right]$. Also, by Lemma \ref{l1}(4), $A$ is adjacent to $\left[\begin{array}{lr}
a & -(1-a)c^{-1}\\
-ac & 1-a
\end{array}\right]$, for every $a\in \mathbb{F}\backslash\{0,1\}$, i.e. $A$ is adjacent to $n-2$ elements from $P_7$. Hence $\deg (A)= |P_5\cup \{E_{11}, x_1\}|+ n-2=|P_5|+2+ n-2=n-1+n=2n-1$ (since $|P_5|=n-1$). Similarly, $\deg (A)=2n-1$, when $A\in P_i$, for $i=4,5,6$. Hence the result follows.
\end{proof}
From Lemma \ref{l1}(3), it is clear that an element  $E$ of $P_7$ is adjacent to only 4 elements from $\displaystyle{\cup_{i=1}^6P_i}$. The following result gives the elements of $P_7$ that are adjacent to $E$.
\begin{lem}
	\label{l2} Let $E=\left[\begin{array}{lr}
		a & b\\
		a(1-a)b^{-1} & 1-a
	\end{array}\right],~E_i=\left[\begin{array}{lr}
		a_i & b_i\\
		a_i(1-a_i)b_i^{-1} & 1-a_i
	\end{array}\right]\in P_7$, for $i=1,2$. Then \begin{enumerate}
	\item $EE_1=\boldsymbol{0}$ if and only if $b_1=-(1-a_1)a^{-1}b$.
	\item $E_2E=\boldsymbol{0}$ if and only if $b_2=-a_2(1-a)^{-1}b$.
\end{enumerate} Moreover, $\deg(E)=2n-1$.
\end{lem}
\begin{proof}(1) Suppose that $EE_1=\boldsymbol{0}$, i.e., $\left[\begin{array}{lr}
		a & b\\
		a(1-a)b^{-1} & 1-a
	\end{array}\right]\left[\begin{array}{lr}
		a_1 & b_1\\
		a_1(1-a_1)b_1^{-1} & 1-a_1
	\end{array}\right]=\boldsymbol{0}$, which gives \\ $\left[\begin{array}{lr}
		aa_1+a_1(1-a_1)bb_1^{-1} & ab_1+b(1-a_1)\\
		a(1-a)a_1b^{-1}+(1-a)a_1(1-a_1)b_1^{-1} & a(1-a)b^{-1}b_1+(1-a)(1-a_1)
	\end{array}\right]=\boldsymbol{0}$, which yields

 $\begin{array}{rc}aa_1+a_1(1-a_1)bb_1^{-1}&=0\\
		ab_1+b(1-a_1)&=0\\ a(1-a)a_1b^{-1}+(1-a)a_1(1-a_1)b_1^{-1}&=0\\ a(1-a)b^{-1}b_1+(1-a)(1-a_1)&=0\end{array}.$\\
	Consequently, $
		b_1=-(1-a_1)a^{-1}b$. 
	Thus, in this case,  $E$ is adjacent to\\ $\left[\begin{array}{lr}
		a_1 & -(1-a_1)a^{-1}b\\
		-a_1ab^{-1} & 1-a_1
	\end{array}\right]$, for each $a_1\in \mathbb{F}\backslash\{0,1\}$, i.e. $E$ has $n-2$ distinct neighbour in this case.\\
(2) Suppose that $E_2E=\boldsymbol{0}$, i.e., $\left[\begin{array}{lr}
		a_2 & b_2\\
		a_2(1-a_2)b_2^{-1} & 1-a_2
	\end{array}\right]\left[\begin{array}{lr}
		a & b\\
		a(1-a)b^{-1} & 1-a
	\end{array}\right]=\boldsymbol{0}$, which gives (by taking transpose) $\left[\begin{array}{lr}
		a & a(1-a)b^{-1}\\
		b & 1-a
	\end{array}\right]\left[\begin{array}{lr}
		a_2 &  a_2(1-a_2)b_2^{-1}\\
		b_2 & 1-a_2
	\end{array}\right]=\boldsymbol{0}$. By (1) above, we have  $\left[\begin{array}{lr}
		a & a(1-a)b^{-1}\\
		b & 1-a
	\end{array}\right]$ is adjacent to $\left[\begin{array}{lr}
		a_2 &  a_2(1-a_2)b_2^{-1}\\
		b_2 & 1-a_2
	\end{array}\right]$ if and only if  $a_2(1-a_2)b_2^{-1}=-(1-a_2)a^{-1}[a(1-a)b^{-1}],\textnormal{ i.e. ~~}  b_2=-a_2(1-a)^{-1}b$.
	Hence $E_2=\left[\begin{array}{lr}
		a_2 & -a_2(1-a)^{-1}b\\
		-(1-a_2)(1-a)b^{-1} & 1-a_2
	\end{array}\right]$ is adjacent to $E$, for each $a_2\in \mathbb{F}\backslash \{0,1\}$, i.e. $E$ has $n-2$ distinct neighbour in this case.
	
	Next, we will determine the degree of $E$. Observe that, if $EE_2=\boldsymbol{0}$, then $E_2$ was also counted in case (1) above. Hence to get distinct $E_2$, we should omit the case when $EE_2=\boldsymbol{0}$ also, i.e. the case when, $EE_2=\boldsymbol{0}=E_2E$. Then $EE_2=\boldsymbol{0}$ gives (by (1) above) $b_2=-(1-a_2)a^{-1}b$ and $E_2E=\boldsymbol{0}$ gives (by (2) above) $b_2=-a_2(1-a)^{-1}b$. Hence $-(1-a_2)a^{-1}b=-a_2(1-a)^{-1}b$, i.e.,$(1-a_2)a^{-1}=a_2(1-a)^{-1}$, which gives $(1-a_2)(1-a)=a_2a$, which yields $a_2=1-a$. Then $b_2=-a_2(1-a)^{-1}b=-(1-a)(1-a)^{-1}b$ giving $b_2=-b$. Thus $EE_2=\boldsymbol{0}$ and $E_2E=\boldsymbol{0}$ if and only if $E_2=\boldsymbol{1}-E$.  
	
Let $\mathbf{A}$ be as given Lemma \ref{l1}(3). Then $\deg(E)= |\mathbf{A}|+n-2+n-2-1=4+2n-5=2n-1$.
\end{proof}

\begin{prop}\label{p1}Let $R=M_2(\mathbb{F})$, where $\mathbb{F}$ is a finite field. Then for any nonzero idempotent $E\in R$, $AE=EA=\boldsymbol{0}$ if and only if $A=\boldsymbol{1}-E$.
\end{prop}
\begin{proof}
	Let $E$ be a nonzero idempotent in $R$ and $P_1,P_2,\cdots, P_7$ be as in Remark \ref{r1}. Clearly $\boldsymbol{1}-E$ is a nonzero idempotent in $R$ such that $E(\boldsymbol{1}-E)=(\boldsymbol{1}-E)E=\boldsymbol{0}$. Since $V(G_{Id}(R))=\dot{\cup}_{i=1}^7P_i$, we have the following cases.\\
	{\bf Case 1)} If $E=E_{11}$, then $\boldsymbol{1}-E=E_{22}$. Let $A$ be a nonzero idempotent in $R$. If $A\in P_3$, then $A=\left[\begin{array}{lr}
		0 & 0\\
		c & 1
	\end{array}\right]$, for some nonzero $c\in \mathbb{F}$. Then $EA=\boldsymbol{0}$, but $AE\ne \boldsymbol{0}$. On the other hand, if $A\in P_4$, then $A=\left[\begin{array}{lr}
	0 & c\\
	0 & 1
\end{array}\right]$, for some nonzero $c\in \mathbb{F}$. Then $AE=\boldsymbol{0}$, but $EA\ne \boldsymbol{0}$. If $\displaystyle{A\in \cup_{i=5}^7P_i}$, then $EA\ne \boldsymbol{0}$ and $AE\ne \boldsymbol{0}$ (by Lemma \ref{l1}(1)). Thus $EA=\boldsymbol{0}$ and $AE=\boldsymbol{0}$ if and only if $A=E_{22}=\boldsymbol{1}-E$.\\
{\bf Case 2)} Suppose that $E=E_{22}$. Hence $\boldsymbol{1}-E=E_{11}$. If $A\in P_5$, then $A=\left[\begin{array}{lr}
1 & 0\\
c & 0
\end{array}\right]$, for some nonzero $c\in \mathbb{F}$. Then $AE=\boldsymbol{0}$, but $EA\ne \boldsymbol{0}$. On the other hand, if $A\in P_6$, then $A=\left[\begin{array}{lr}
1 & c\\
0 & 0
\end{array}\right]$, for some nonzero $c\in \mathbb{F}$. Then $EA=\boldsymbol{0}$, but $AE\ne \boldsymbol{0}$. If $A\in P_3\cup P_4\cup P_7$, then $EA\ne \boldsymbol{0}$ and $AE\ne \boldsymbol{0}$ (by Lemma \ref{l1}(2)). Thus $EA=\boldsymbol{0}$ and $AE=\boldsymbol{0}$ if and only if $A=E_{11}=\boldsymbol{1}-E$.\\
{\bf Case 3)} Suppose that $E\in P_3$, i.e. $E=\left[\begin{array}{lr}
	0 & 0\\
	c & 1
\end{array}\right]$, for some nonzero $c\in \mathbb{F}$. Hence $\boldsymbol{1}-E=\left[\begin{array}{lr}
1 & 0\\
-c & 0
\end{array}\right]\in P_5$. Observe that $E_{11}E=\boldsymbol{0}$, but $EE_{11}\ne \boldsymbol{0}$; $EE_{22}\ne \boldsymbol{0}$ and $E_{22}E\ne \boldsymbol{0}$. If $A\in (P_3\cup P_4\cup P_6\cup P_7)\backslash \{x_1,x_2\}$, where $x_1=\left[\begin{array}{lr}
0 & -c^{-1}\\
0 & 1
\end{array}\right],~ x_2=\left[\begin{array}{lr}
a & -(1-a)c^{-1}\\
-ac & 1-a
\end{array}\right]$, for any nonzero $a\in \mathbb{F}$, then $EA\ne \boldsymbol{0}$ and $AE\ne \boldsymbol{0}$. Also, $x_1E\ne \boldsymbol{0}$ and $x_2E\ne \boldsymbol{0}$.  
If $A\in P_5$, then $AE=\boldsymbol{0}$. But $EA\ne \boldsymbol{0}$ if $A\ne \boldsymbol{1}-E$. Thus , in this case also, $EA=\boldsymbol{0}$ and $AE=\boldsymbol{0}$ if and only if $A=E_{11}=\boldsymbol{1}-E$.

Similarly, the result follows when $E\in P_i$, for $i=4,5,6$.

If $A\in P_7$, then the result follows from the second-last para of proof of Lemma \ref{l2}.
\end{proof}
Next we give the structure of the idempotent graph $I(R)$.
\begin{thm}
	Let $R=M_2(\mathbb{F})$. Then $I(R)$ is a disjoint union of $\frac{n(n+1)}{2}$ copies of $K_2$.
\end{thm}
\begin{proof}
	The proof follows from Theorem \ref{t1} and Proposition \ref{p1}.
\end{proof}
We close the section by giving the structure of $G_{Id}(R)$. 
\begin{thm}\label{t2}Let $R=M_2(\mathbb{F})$, where $\mathbb{F}$ is a finite field. Then $G_{Id}(R))$ is a connected regular graph of degree $2n-1$. Moreover, $diam(G_{Id}(R))=2$ and $$gr(G_{Id}(R))=\left\{\begin{array}{cl}
		4, & \textnormal{ if } \mathbb{F}=\mathbb{Z}_2\\
		3, & \textnormal{ otherwise }
	\end{array}\right..$$
\end{thm}
\begin{proof} Let $P_1, P_2, \cdots P_7$ be the sets as in Remark \ref{r1}. It is clear that $V(G_{Id}(R))=\dot{\cup}_{i=1}^7P_i$. From Lemma \ref{l1} it is clear that a vertex $E_{11}$ is connected to every other vertex by a path of length at most 2, hence  $G_{Id}(R))$ is  connected. Also, from Lemma \ref{l1}(1),(2), Corollary \ref{c1} and Lemma \ref{l2}, it follows that $\deg(v)=2n-1$, $\forall v\in V(G_{Id}(R))$. Therefore  $G_{Id}(R))$ is a regular graph of degree $2n-1$. 
	
Now by Lemma \ref{l1}, it is evident that $E_{11}$ and $E_{22}$ have no common neighbour in $V(G_{Id}(R))$. It also follows that no two adjacent vertices in $\displaystyle{\cup_{i=1}^6P_i}$ have common neighbour in $\displaystyle{\cup_{i=1}^6P_i}$. Suppose that $\mathbb{F}=\mathbb{Z}_2$. By Remark \ref{r1}, $|P_7|=0$, i.e., $P_7=\emptyset$. Let $A=\left[\begin{array}{lr}
		1 & 0\\
		1 & 0
\end{array}\right], B=\left[\begin{array}{lr}
0 & 0\\
1 & 1
\end{array}\right]$. Then $AE_{22}=AB=E_{11}B=E_{11}E_{22}=\boldsymbol{0}$, hence $E_{11}\sim E_{22}\sim A\sim B\sim E_{11}$ is a 4-cycle in $G_{Id}(R)$.
Suppose that $\mathbb{F}\ne \mathbb{Z}_2$. Then $P_7$ is a non-empty set. Let $A_1=\left[\begin{array}{lr}
	0 & 0\\
	c & 1
\end{array}\right]\in P_3, A_2=\left[\begin{array}{lr}
a & -(1-a)c^{-1}\\
-ac & 1-a
\end{array}\right]\in P_7$ and $A_3=\left[\begin{array}{lr}
1 & 0\\
a(1-a)^{-1}c & 0
\end{array}\right]\in P_5$. Then $A_1\sim A_2\sim A_3\sim A_1$ is a 3-cycle in $G_{Id}(R)$ (since $A_1A_2=A_2A_3=A_3A_1=\boldsymbol{0}$). Hence the result of girth follows.
 
 From Lemma \ref{l1}, it follows that any two vertices in $\displaystyle{\cup_{i=1}^6 P_i}$ are either adjacent or have at least one common neighbour from $\displaystyle{\cup_{i=1}^6 P_i}$. Also, by Lemma \ref{l1}, it is clear that a vertex in $\displaystyle{\cup_{i=1}^6 P_i}$ is adjacent to any vertex in $P_7$ by a path of length at most 2. Hence to prove the result, it remains to show that any two non-adjacent vertices in $P_7$ are connected by a path of length two.  Let $w_i=\left[\begin{array}{lr}
 	a_i & b_i\\
 	a_i(1-a_i)b_i^{-1} & 1-a_i
 \end{array}\right]$, for $i=1,2$, be two elements in $P_7$ that are non-adjacent in $G_{Id}(R)$. We will show that there exists $z=\left[\begin{array}{lr}
 	a & b\\
 	a(1-a)b^{-1} & 1-a
 \end{array}\right]\in P_7$ such that $w_1z=\boldsymbol{0}$ and $zw_2=\boldsymbol{0}$.
 Let $zw_2=\boldsymbol{0}$. By Lemma \ref{l2}, we get $b_2=-(1-a_2)a^{-1}b$, which yields $b=-a(1-a_2)^{-1}b_2$. Whereas $w_1z=\boldsymbol{0}$ gives (by Lemma \ref{l2} again)  $b=-(1-a)a_1^{-1}b_1$. Hence we get $-(1-a)a_1^{-1}b_1=-a(1-a_2)^{-1}b_2$, which gives $a(1-a)^{-1}=(1-a_2)a_1^{-1}b_1b_2^{-1}$.\\ Thus choose $a,b$ and $z$ such that $a(1-a)^{-1}=(1-a_2)a_1^{-1}b_1b_2^{-1}$, $b=-a(1-a_2)^{-1}b_2$ and $z=\left[\begin{array}{lr}
 	a & b\\
 	a(1-a)b^{-1} & 1-a
 \end{array}\right]$. Then $z\in P_7$ and $w_1\sim z\sim w_2$ is a path in $G_{Id}(R)$(since $w_1z=zw_2=\boldsymbol{0}$). Hence $d(w_1,w_2)=2$. 
 This completes the proof.
 \end{proof}
  \section{Wiener and Harary Index}  
  The \textit{Wiener index} of the graph $G$, denoted by $W(G)$, is defined to be the sum of all distanced between any two vertices of $G$. Let $d_G(v)$ denote the sum of distances of the vertex $v$ from all the vertices of $G$, then the Wiener index can be redefined as $W(G)=\frac{1}{2}\sum_{v\in V(G)}d_G(v)$.
  The {\it Wiener index} of a graph $G$ was first introduced by Wiener \cite{WIE} in 1947. The Wiener index $W(G)$ is an oldest topological index. The \textit{Harary index} $H(G)$ of a graph $G$, has been introduced independently by Plav\v{s}ić et al. \cite{PLA} and by Ivanciuc et al. \cite{IVA} in 1993. The \textit{Harary index} of a graph $G$ is defined as: $\displaystyle{H(G)=\sum_{u,v\in V}\frac{1}{d(u,v)}}$, where the summation runs over all unordered pairs of vertices of   graph $G$(see \cite{KEX}).	 Let $d'_G(v)$ denote the sum of reciprocal of distances of the vertex $v$ from all the vertices of $G$, then the Harary index can be redefined as $\displaystyle{H(G)=\sum_{v\in V(G)}d'_G(v)}$.

 In this section we determine the Wiener index and Harary index of $G_{Id}(R)$.

\begin{prop}\label{p2}
Let $G$ be regular graph of degree $r$ and $diam(G)=2$. Then the number of vertices at a distance $2$ from a vertex $v$ are $|V(G)|-r-1$.
\end{prop}
\begin{proof}Let $N(v)=\{u\in V(G)~|~\textnormal{ if } uv\in E(G) \}$. Since $G$ is regular graph of degree $r$ and $diam(G)=2$, a vertex $u$ is at a distance 2 from the vertex $v$ if and only if $u\in V(G)\backslash \Big(N(v)\cup \{v\}\Big)$. Since $\deg(v)=|N(v)|$, the result follows.
\end{proof}
The following result gives the sum of the distances and the sum of the reciprocal of distances of each vertex from other vertices in $G_{Id}(R)$.
\begin{coro}\label{c2}Let $R=M_2(\mathbb{F})$, where $\mathbb{F}$ is a finite field. Then for each nontrivial idempotent in $R$, $d_{G_{Id}(R)}(A)=2n^2-1$ and $d'_{G_{Id}(R)}(A)=\frac{1}{2}(n^2+3n-2)$.	
\end{coro}
\begin{proof}By Theorem \ref{t1} and Theorem \ref{t2}, $G_{Id}(R)$ is a connected, regular graph of degree $2n-1$ on $n^2+n$ vertices and $diam(G_{Id}(R))=2$. Hence for each nontrivial idempotent $A\in R$, the number of vertices at a distance 2 from $A$ are $n^2+n-(2n-1)-1=n^2-n$. Therefore $d_{G_{Id}(R)}(A)=\deg(A)+$2(the number of vertices at a distance 2 from $A$)
	$=2n-1+2(n^2-n)=2n^2-1$. On the other hand, $d'_{G_{Id}(R)}(A)=\deg(A)+\frac{1}{2}$(the number of vertices at a distance 2 from $A$)
	$=2n-1+\frac{1}{2}(n^2-n)=\frac{1}{2}(n^2+3n-2)$.
\end{proof}
We conclude by giving the Wiener index and Harary index of $G_{Id}(R)$.
\begin{thm}
	Let $R=M_2(\mathbb{F})$, where $\mathbb{F}$ is a finite field. Then $W(G_{Id}(R))=\frac{1}{2}(n^2+n)(2n^2-1)$ and $H(G_{Id}(R))=\frac{1}{2}(n^2+n)(n^2+3n-2)$.
\end{thm}
\begin{proof}The Wiener index of $G_{Id}(R)$ is given by \\
	$\displaystyle{\begin{array}{ll}
			W(G_{Id}(R))&=\frac{1}{2}\left(\sum\limits_{A\in V(G_{Id}(R))}^{}d_{G_{Id}(R)}(A)\right)\\
			& = \frac{1}{2}\left(|V(G_{Id}(R))|\times d_{G_{Id}(R)}(A)\right) \hspace{2cm} \cdots \textnormal{ by Corollary } \ref{c2}\\
			& = \frac{1}{2}(n^2+n)(2n^2-1).
	\end{array}}$

The Harary index of $G_{Id}(R)$ is given by \\
$\displaystyle{\begin{array}{ll}
		H(G_{Id}(R))&=\sum\limits_{A\in V(G_{Id}(R))}^{}d'_{G_{Id}(R)}(A)\\
		& = |V(G_{Id}(R))|\times d'_{G_{Id}(R)}(A) \hspace{2.8cm} \cdots \textnormal{ by Corollary } \ref{c2}\\
		& = \frac{1}{2}(n^2+n)(n^2+3n-2).
\end{array}}$

\end{proof}

\end{document}